\documentclass{article}

\usepackage{amsmath, amsthm, amsfonts, mathrsfs, amssymb,latexsym}
\usepackage{hyperref}
\usepackage[pdftex]{graphicx}

\theoremstyle{plain}
\newtheorem{theorem}{Theorem}

\begin{document}
\title{Viviani Polytopes and Fermat Points}
\author{Li Zhou (lzhou@polk.edu)}
\date{}
\maketitle

Get some unit vectors.

Assemble them so that they sum to the zero vector. Such arrangements are easily formed in the plane (or in space) using mechanical linkages. They are even easier to create in higher dimensions. Figures 1a and 2a give two planar examples. 

Attach to each unit vector a perpendicular line, that is, make each vector into a unit normal. Shift these lines in the plane in any way you please, while keeping the directions of their unit normals. Figures 1b and 1c are two examples. 

\begin{center}
\includegraphics*[viewport=162 348 448 443]{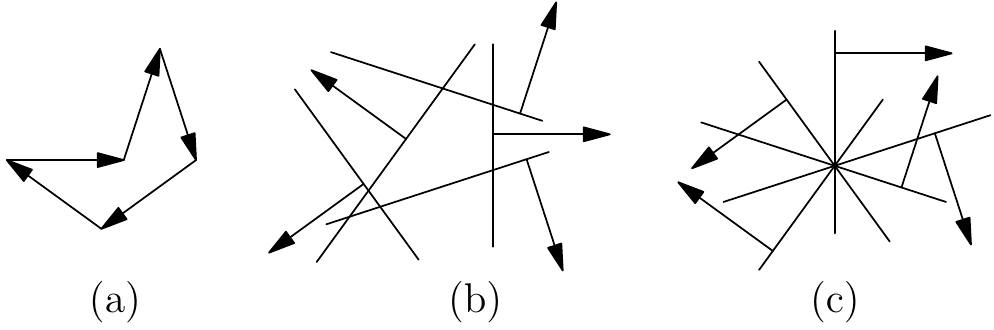}\\*\textbf{Figure 1.}
\end{center}

The resulting configurations have a magical property: the sum of the signed distances from any point to the lines is constant! 

In this paper we investigate this in the general setting of Euclidean $n$-space, where the unit vectors are now normal to hyperplanes. The result is a simple characterization that answers a question posed in this \textsc{Journal} by Elias Abboud \cite{ea}, and links the property to the Fermat point.   

\section*{The characterization}

Let $\mathcal{P}=\{p_1, \ldots, p_k\}$ be a set of oriented hyperplanes in $\mathbb{R}^n$, not necessarily distinct, where the orientation of each $p_i$ is specified by a unit normal $\mathbf{n}_i$. For any point $P\in\mathbb{R}^n$, let $d_i=d(P,p_i)$ be the signed distance from $P$ to $p_i$, with $d_i<0$ if $\mathbf{n}_i$ points to the half-space containing $P$, and $d_i>0$ if $\mathbf{n}_i$ points in the opposite direction. Define $v: \mathbb{R}^n\to \mathbb{R}$ by $v(P)=\sum_{i=1}^k d(P,p_i)$ for all points $P\in\mathbb{R}^n$. We call the set $\mathcal{P}$ \emph{Viviani} if $v$ is a constant (in Figure 1b, the constant is approximately 2.3; in 1c, it is 0). The following theorem characterizes this property.

\begin{theorem}\label{th1} $\mathcal{P}=\{p_1, \ldots, p_k\}$ is Viviani if and only if $\mathbf{n}_1+\cdots+\mathbf{n}_k=\mathbf{0}$.
\end{theorem}

\begin{proof} For any point $X\in\mathbb{R}^n$, we denote by $\mathbf{X}$ the vector from a fixed origin $O$ to $X$. Fix a point $P_i\in p_i$ for each $i=1, \ldots, k$. Then for any $P\in\mathbb{R}^n$, $$v(P)=\sum_{i=1}^k (\mathbf{P}_i-\mathbf{P})\cdot\mathbf{n}_i=\sum_{i=1}^k \mathbf{P}_i\cdot\mathbf{n}_i- \mathbf{P}\cdot\sum_{i=1}^k \mathbf{n}_i.$$ Hence $v$ is a constant if and only if $\sum_{i=1}^k \mathbf{n}_i=\mathbf{0}$.
\end{proof}

The proof reveals that $v$ is an affine function with $\nabla v=-\sum_{i=1}^k \mathbf{n}_i$. Hence $\mathcal{P}$ is Viviani if $v$ assumes the same value at $n+1$ affinely independent points (that is, points not lying in a common hyperplane in $\mathbb{R}^n$). Furthermore, if $\mathcal{P}$ is not Viviani, then the level sets of $v$ are parallel hyperplanes in $\mathbb{R}^n$ orthogonal to $\sum_{i=1}^k \mathbf{n}_i$.

\section*{Examples}

A convex polytope is a finite region of $\mathbb{R}^n$ enclosed by a finite number of hyperplanes. Let $\mathcal{P}=\{p_1,  \ldots, p_k\}$ be the $k$ distinct hyperplanes enclosing a convex polytope, and let $\mathbf{n}_i$ be the outward unit normal to each $p_i$. Then we call a convex polytope \emph{Viviani}, if its associated $\mathcal{P}$ is Viviani. For $n=2$ and $3$, we shall refer to such polytopes as Viviani polygons and polyhedra. The following examples are easily verified using Theorem \ref{th1}.

\medskip 
\noindent\textbf{Example 1.} The only Viviani triangles are equilateral triangles.

\medskip 
\noindent\textbf{Example 2.} The only Viviani quadrilaterals are parallelograms.

\medskip 
\noindent\textbf{Example 3.} Equiangular polygons of any number of sides are Viviani (Figure 1b contains an equiangular pentagon.).

\medskip

\noindent\textbf{Example 4.} The regular (Platonic) polyhedra are Viviani.

\medskip

If $n=2$ and $k\ge 5$, or if $n\ge 3$ and $k\ge 4$, then there is more freedom for $k$ unit vectors in $\mathbb{R}^n$ to sum to $\mathbf{0}$. Viviani polytopes in these cases are no longer very special geometrically, other than the characterization in Theorem \ref{th1}. In fact, there are already infinitely many irregular Viviani tetrahedra.

\medskip 
\noindent\textbf{Example 5.} For $0<t<\pi$, a tetrahedron whose faces have the unit normals 
 $\left\langle \cos \frac{t}{2}, \frac{\sin t}{2}, \frac{1-\cos t}{2}\right\rangle $, $\left\langle -\cos \frac{t}{2}, \frac{\sin t}{2}, \frac{1-\cos t}{2}\right\rangle $, $\left\langle 0,-\sin t,\cos t\right\rangle $, and $\left\langle 0,0,-1\right\rangle $ is Viviani.
 
\section*{Duality}

Let $P_1, \ldots, P_k$ be non-collinear points in $\mathbb{R}^n$, then there is a unique point $P$ in their convex hull, known as the Fermat point, such that the sum of the distances from $P$ to $P_1, \ldots, P_k$ is least possible. If the Fermat point $P$ is distinct from $P_1,\ldots, P_k$, then it is easy to see that $\mathbf{n}_1+\cdots+\mathbf{n}_k=\mathbf{0}$, where $\mathbf{n}_i$ is the unit vector pointing from $P$ to $P_i$ (see Figure 2, for a proof, see \cite{km}). This leads to the following duality.

\begin{center}
\includegraphics*[viewport=181 340 429 452]{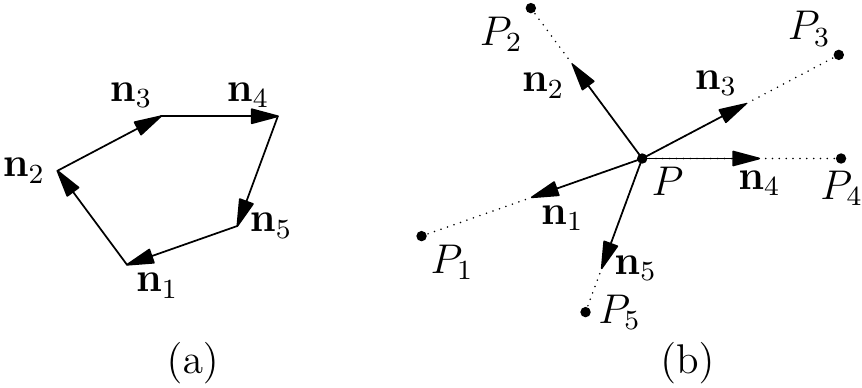}\\*\textbf{}\textbf{Figure 2.}
\end{center}

\begin{theorem}\label{th2} (a) Suppose that $P$ is the Fermat point of, and distinct from, given points $P_1,\ldots, P_k$ in $\mathbb{R}^n$. Let $\mathbf{n}_i$ be the unit vector pointing from $P$ to $P_i$, and $p_i$ be an oriented hyperplane with unit normal $\mathbf{n}_i$. Then  $\mathcal{P}=\{p_1, \ldots, p_k\}$ is Viviani. (b) Conversely, let  $P$ be a fixed point and $\mathcal{P}=\{p_1, \ldots, p_k\}$ be a set of oriented hyperplanes with unit normals $\mathbf{n}_1,\ldots, \mathbf{n}_k$, such that all $d(P, p_i)\ge 0$, or all $d(P, p_i)\le 0$. If $\mathcal{P}$ is Viviani and $P_i$ is the projection of $P$ onto $p_i$, then $P$ is the Fermat point of $P_1,  \ldots, P_k$. 
\end{theorem}

\begin{proof} (a) By assumption, $\mathbf{n}_1+\cdots+\mathbf{n}_k=\mathbf{0}$. Hence $\mathcal{P}$ is Viviani by Theorem \ref{th1}.

(b) Let $Q$ be any point different from $P$. Since $\mathcal{P}$ is Viviani, $v(P)=v(Q)$. Hence,
$$\sum_{i=1}^k \left\|\mathbf{P}_i-\mathbf{P}\right\|=\sum_{i=1}^k |d(P, p_i)|=|v(P)|=|v(Q)|<\sum_{i=1}^k \left\|\mathbf{P}_i-\mathbf{Q}\right\|.$$
\end{proof}

\section*{History}

For a triangle $ABC$, the problem of finding the Fermat point of $A,B,C$ was first proposed by Fermat to Torricelli in a private letter. Torricelli solved the problem and his solution was published by his student Viviani in 1659 \cite[Appendix, pp.\ 143--150]{v}. Torricelli's solution uses the fact that the sum of the distances from any point inside an equilateral triangle to its sides is constant, which is commonly known today as Viviani's theorem. However, it is clear from \cite{v} that Viviani proved a bit more, namely that (a) the sum of the distances from any point inside a regular polygon to its sides is constant, and is less than the sum from any point outside the regular polygon; (b) if $A_1\cdots A_k$ is a regular polygon with center $A$, and $B$ is a point different from $A$, then the sum of the distances from $A$ to $A_i$ is less than the sum of the distances from $B$ to $A_i$; and (c) if $B_i$ is a point on the segment $AA_i$, then $A$ is the Fermat point of $B_1, \ldots, B_k$.

Viviani proved (a) by considering areas. For (b) he drew through each $A_i$ the line $p_i$ perpendicular to $AA_i$, and then applied (a) to the regular polygon $\mathcal{P}=\{p_1, \ldots, p_k\}$, as in our proof of Theorem \ref{th2}. Finally, Viviani proved (c) using (b) and the triangle inequality: 
$$\sum_{i=1}^k AB_i=\sum_{i=1}^k (AA_i-B_iA_i)<\sum_{i=1}^k BA_i-\sum_{i=1}^k B_iA_i<\sum_{i=1}^k BB_i.$$

\paragraph{Acknowledgments.} I would like to thank the referees and the editor for helpful suggestions which improved the mathematical accuracy and the presentation of this paper.  

\paragraph{Summary.} Given a set of oriented hyperplanes $\mathcal{P}=\{p_1, ..., p_k\}$ in $\mathbb{R}^n$, define $v(P)$ for any point $P\in\mathbb{R}^n$ as the sum of the signed distances from $P$ to $p_1$,  \ldots, $p_k$. We give a simple geometric characterization of  $\mathcal{P}$ so that $v$ is a constant. The characterization leads to a connection with the Fermat point of $k$ points in $\mathbb{R}^n$. Finally, we discuss historically the full content of Viviani's theorem.

\end{document}